\documentclass[a4paper,11pt]{amsart}

\usepackage{tikz,mathrsfs}
\usepackage{mathtools}

\usetikzlibrary{arrows}

\tikzset{>=stealth'}

\def\arrowLengthDisplayStyle{4ex}
\def\arrowHeightDisplayStyle{.8ex}
\def\arrowSkipDisplayStyle{.5ex}
\def\arrowLengthTextStyle{3ex}
\def\arrowHeightTextStyle{.8ex}
\def\arrowSkipTextStyle{.4ex}
\def\arrowLengthScriptStyle{2.5ex}
\def\arrowHeightScriptStyle{.6ex}
\def\arrowSkipScriptStyle{.3ex}
\def\arrowLengthScriptScriptStyle{2ex}
\def\arrowHeightScriptScriptStyle{.4ex}
\def\arrowSkipScriptScriptStyle{.2ex}

\renewcommand{\to}{\arrow{->}}

\newcommand{\embed}{\arrow{right hook->}}
\renewcommand{\mapsto}{\arrow{|->}}

\newcommand{\MakeTikzArrowWithSuperscriptSubscript}[4]
 {
  \mathchoice
   { 
    \hspace*{\arrowSkipDisplayStyle}
    \begin{tikzpicture}[baseline]
     \draw [#1] (0,\arrowHeightDisplayStyle) -- node [above] {$#2$} node [below] {$#3$} (#4 * \arrowLengthDisplayStyle, \arrowHeightDisplayStyle);
    \end{tikzpicture}
    \hspace*{\arrowSkipDisplayStyle}
   }
   { 
    \hspace*{\arrowSkipTextStyle}
    \begin{tikzpicture}[baseline]
     \draw [#1] (0,\arrowHeightTextStyle) -- node [above] {$\scriptstyle #2$} node [below] {$\scriptstyle #3$} (#4 * \arrowLengthTextStyle, \arrowHeightTextStyle);
    \end{tikzpicture}
    \hspace*{\arrowSkipTextStyle}
   }
   { 
    \hspace*{\arrowSkipScriptStyle}
    \begin{tikzpicture}[baseline]
     \draw [#1] (0,\arrowHeightScriptStyle) -- node [above] {$\scriptscriptstyle #2$} node [below] {$\scriptscriptstyle #3$} (#4 * \arrowLengthScriptStyle, \arrowHeightScriptStyle);
    \end{tikzpicture}
    \hspace*{\arrowSkipScriptStyle}
   }
   { 
    \hspace*{\arrowSkipScriptScriptStyle}
    \begin{tikzpicture}[baseline]
     \draw [#1] (0,\arrowHeightScriptScriptStyle) -- node [above] {$\scriptscriptstyle #2$} node [below] {$\scriptscriptstyle #3$} (#4 * \arrowLengthScriptScriptStyle, \arrowHeightScriptScriptStyle);
    \end{tikzpicture}
    \hspace*{\arrowSkipScriptScriptStyle}
   }
 }

\newcommand{\MakeTikzArrowWithCentralLabel}[3]
 {
  \mathchoice
   { 
    \hspace*{\arrowSkipDisplayStyle}
    \begin{tikzpicture}[baseline]
     \draw [#1] (0,\arrowHeightDisplayStyle) -- node [fill=white,inner sep=1pt] {$#2$} (#3 * \arrowLengthDisplayStyle, \arrowHeightDisplayStyle);
    \end{tikzpicture}
    \hspace*{\arrowSkipDisplayStyle}
   }
   { 
    \hspace*{\arrowSkipTextStyle}
    \begin{tikzpicture}[baseline]
     \draw [#1] (0,\arrowHeightTextStyle) -- node [fill=white,inner sep=1pt] {$\scriptstyle #2$} (#3 * \arrowLengthTextStyle, \arrowHeightTextStyle);
    \end{tikzpicture}
    \hspace*{\arrowSkipTextStyle}
   }
   { 
    \hspace*{\arrowSkipScriptStyle}
    \begin{tikzpicture}[baseline]
     \draw [#1] (0,\arrowHeightScriptStyle) -- node [fill=white,inner sep=1pt] {$\scriptscriptstyle #2$} (#3 * \arrowLengthScriptStyle, \arrowHeightScriptStyle);
    \end{tikzpicture}
    \hspace*{\arrowSkipScriptStyle}
   }
   { 
    \hspace*{\arrowSkipScriptScriptStyle}
    \begin{tikzpicture}[baseline]
     \draw [#1] (0,\arrowHeightScriptScriptStyle) -- node [fill=white,inner sep=1pt] {$\scriptscriptstyle #2$} (#3 * \arrowLengthScriptScriptStyle, \arrowHeightScriptScriptStyle);
    \end{tikzpicture}
    \hspace*{\arrowSkipScriptScriptStyle}
   }
 }

\def\arrow#1{\def\lastArrowStyle{#1}
             \futurelet\testchar\arrowMaybeStreched}
\def\arrowMaybeStreched{\ifx[\testchar \let\next\arrowStreched
                         \else \let\next\arrowUnstreched \fi
                        \next}

\def\arrowStreched[#1]{\def\lastArrowStrech{#1}
                       \futurelet\testchar\arrowMaybeLabel}
\def\arrowUnstreched{\def\lastArrowStrech{1}
                     \futurelet\testchar\arrowMaybeLabel}

\def\arrowMaybeLabel{\ifx^\testchar \let\next\arrowSuperscript
                      \else \ifx_\testchar \let\next\arrowSubscript
                             \else \ifx~\testchar \let\next\arrowCentralLabel
                                    \else \let\next\arrowNoLabel
                                   \fi
                            \fi
                     \fi
                     \next}

\def\arrowSuperscript^#1{\def\lastArrowSuperscript{#1}
                         \futurelet\testchar\arrowSuperMaybeSub}
\def\arrowSuperMaybeSub{\ifx_\testchar \let\next\arrowSuperscriptSubscript
                         \else \let\next\arrowSuperscriptNoSubscript \fi
                        \next}

\def\arrowSubscript_#1{\def\lastArrowSubscript{#1}
                         \futurelet\testchar\arrowSubMaybeSuper}
\def\arrowSubMaybeSuper{\ifx^\testchar \let\next\arrowSubscriptSuperscript
                         \else \let\next\arrowSubscriptNoSuperscript \fi
                        \next}

\def\arrowSuperscriptSubscript_#1{\def\lastArrowSubscript{#1}
                                  \arrowDrawSupSub}
\def\arrowSuperscriptNoSubscript{\def\lastArrowSubscript{}
                                 \arrowDrawSupSub}
\def\arrowSubscriptSuperscript^#1{\def\lastArrowSuperscript{#1}
                                  \arrowDrawSupSub}
\def\arrowSubscriptNoSuperscript{\def\lastArrowSuperscript{}
                                 \arrowDrawSupSub}
\def\arrowNoLabel{\def\lastArrowSuperscript{}
                  \def\lastArrowSubscript{}
                  \arrowDrawSupSub}

\def\arrowCentralLabel~#1{\MakeTikzArrowWithCentralLabel{\lastArrowStyle}{#1}{\lastArrowStrech}}
\def\arrowDrawSupSub{\MakeTikzArrowWithSuperscriptSubscript{\lastArrowStyle}{\lastArrowSuperscript}{\lastArrowSubscript}{\lastArrowStrech}}

\usepackage{amssymb,tikz,mathrsfs}
 
\newtheorem{theorem}{Theorem}

\newtheorem{lemma}[theorem]{Lemma}
\newtheorem{proposition}[theorem]{Proposition}

\theoremstyle{definition}
\newtheorem{construction}[theorem]{Construction}
\newtheorem{definition}[theorem]{Definition}
\newtheorem{example}[theorem]{Example}

\newtheorem{remark}[theorem]{Remark}

\DeclareMathOperator{\End}{End}

\DeclareMathOperator{\id}{id}

\DeclareMathOperator{\op}{op}

\newcommand{\iso}{\cong}

\renewcommand{\geq}{\geqslant}

\DeclareMathOperator{\Cone}{Cone}

\newcommand{\R}{\mathbf{R}\!}
\newcommand{\tail}{{\rm t}}
\newcommand{\start}{{\rm s}}

\DeclareMathOperator{\red}{red}
\DeclareMathOperator{\dec}{dec}

\renewcommand{\H}{\operatorname{H}}
\newcommand{\m}{_{\rm M}}

\newcommand{\ci}[1]{\chi(#1)}
\newcommand{\st}[1]{\operatorname{s}_{#1}}
\newcommand{\ta}[1]{\operatorname{t}_{#1}}

\title{Quivers for silting mutation}

\author[Oppermann]{Steffen Oppermann}

\thanks{This project began in discussion with Hugh Thomas, and much of the work presented here was done when visiting UNB. I would like to thank Hugh for inviting me to Fredericton and for the interesting discussions we had there.}

\thanks{Supported by NRF grant 221893.}

\address{Steffen Oppermann \\ Department of Mathematical Sciences \\ NTNU \\ 7491 Trondheim \\ Norway}

\numberwithin{theorem}{section}

\allowdisplaybreaks

\begin{document}

\maketitle

\begin{abstract}
We give a combinatorial mutation rule for Aihara's and Iyama's silting mutation.

As an application, we reprove Keller-Yang's mutation rule for Ginzburg algebras, and obtain an analog of that rule for arbitrary dimension.
\end{abstract}

\section{introduction}

Keller and Vossieck introduced the notion of \emph{silting objects} in \cite{keller-vossieck}, as a means to studying aisles in derived categories. These are objects generating the triangulated category, such that all positive self-extensions vanish. One may note that, by Keller's version of the derived Morita theorem (see \cite{keller-derived-morita}), if one has a silting object $T$ in an algebraic triangulated category, then this triangulated category is equivalent to the perfect complexes over the derived endomorphism ring of $T$.

Aihara and Iyama \cite{aihara-iyama}, inspired by cluster mutation, introduced a mutation procedure for silting objects: They remove a summand from the silting object, and replace it by a different one, using an approximation by the other summands.

In \cite{buan-reiten-thomas}, it is proven that certain silting objects correspond to cluster tilting objects in (higher) cluster categories, and moreover that this identification translates silting mutation to cluster tilting mutation. 

\medskip
Our aim here is to give a combinatorial mutation rule for silting mutation. We assume the dg-algebra we start with is given as a quiver (with arrows in all non-negative degrees) -- see Construction~\ref{const.usual_algebra} for a justification of this assumption. Then we show:

\begin{theorem} \label{thm.main_result}
Let $(kQ, d)$ be a dg quiver algebra, concentrated in (homologically) non-negative degree. Let $i$ be a vertex of $Q$ with no loops of degree $0$ attached to it.

Then the derived endomorphism ring of the left Aihara-Iyama silting mutation of $kQ$ at $i$ is given as the mutated quiver $(kM, \partial)$ as described in Section~\ref{sec.mutation}.
\end{theorem}

An obvious dual involving right silting mutation also holds. However, since this is the exact same statement for the opposite quiver, we restrict to considering left silting mutation here.

\medskip
While the mutation rule may seem combinatorially quite involved, we illustrate in Section~\ref{sec.examples} on some small examples that it is readily applicable.

In Section~\ref{sec.ginzburg} we apply our mutation rule to Ginzburg dg-algebras, obtaining immediately the result of Keller and Yang \cite{keller-yang} that the silting mutation is precisely the Ginzburg dg-algebra of the mutated quiver.

We then proceed (in Section~\ref{sec.higher_ginz}) to generalize this result to arbitrary dimension. In \cite{van_den_bergh}, Van den Bergh proves that (under certain assumptions) any Calabi-Yau algebra is quasi-isomorphic to one given by a higher quiver with potential. Here we obtain a mutation rule for these higher dimensional quivers with potential, which we show to coincide with the endomorphism ring of a silting object.

\medskip
Combinatorial mutation rules are already known for cluster categories (and various generalizations). We hope that future work will show that these cluster mutation rules are (in some sense) induced by the silting mutation rule presented here.

\section{Setup} \label{sect.setup}

Throughout this paper, we denote by $\Lambda$ a (homologically) non-negatively graded dg-algebra over $R = k^n$, where $k$ is a field. We assume moreover that, as graded algebra, $\Lambda$ is free on a degree-wise finite set of generators, that is $\Lambda \iso {\mathrm T}_R A$ for some degree-wise finite dimensional $R$-$R$-bimodule $A$.

Choosing generators for $A$, we see that $\Lambda$ is given as a quiver algebra $kQ$ (with arrows allowed in all non-negative degrees), together with a differential. We will denote the arrows of $Q$ by small greek letters $\varphi$, $\psi$, \dots, and their degrees by $|\varphi|$, $|\psi|$, \dots . It clearly suffices to know the differential $d$ on the arrows -- its value on paths can then be calculated using the graded Leibnitz rule.

\begin{example}
Let $A$ be a basic finite dimensional hereditary $k$-algebra over an algebraically closed field $k$. Then $\Lambda=A$, considered as dg-algebra concentrated in degree $0$, satisfies our assumptions.
\end{example}

\begin{construction} \label{const.usual_algebra}
Let $A$ be any basic finite dimensional $k$-algebra over an algebraically closed field $k$. Then $A$ is given by a quiver with relations, say $A = k Q^{(0)} / (R)$. One may pick a minimal set of relations, and consider $Q^{(1)}$ given by adding to $Q^{(0)}$ arrows in degree $1$ corresponding to these relations, and a differential, such that $\H_0(kQ^{(1)}) = A$. Now pick a generating set for $\H_1(kQ^{(1)})$, and add arrows in degree $2$ killing these generators in homology. This gives a new quiver $Q^{(2)}$ with differential, such that $\H_1(kQ^{(2)}) = 0$ and $\H_0(kQ^{(2)}) = A$. Iterating this one obtains a quiver $Q$ such that $kQ$ is quasi-isomorphic to $A$.
\end{construction}

\begin{example} \label{ex.A_3/rad^2}
Let $A = k[1 \to^{\varphi} 2 \to^{\psi} 3] / (\psi \varphi)$. Then $A$ is quasi-isomorphic to
\[ \Lambda = k[
\begin{tikzpicture}[baseline=-3.5pt]
 \node (1) at (0,0) {$1$};
 \node (2) at (.8,0) {$2$};
 \node (3) at (1.6,0) {$3$};
 \draw [->] (1) to node [above] {$\scriptstyle \varphi$} (2);
 \draw [->] (2) to node [above] {$\scriptstyle \psi$} (3);
 \draw [->, bend right=30] (1) to node [below] {$\scriptstyle \omega$} (3); 
\end{tikzpicture} ], \]
where $\omega$ is in degree $1$ and $d \omega = \psi \varphi$.
\end{example}

\begin{remark}
Let $\mathscr{T}$ be an algebraic triangulated category, and $S \in \mathscr{T}$ a silting complex. Then $\R\End_{\mathscr{T}}(S)$ is a non-negatively graded dg-algebra. However it is not clear how to turn this dg-algebra into a quiver in a reasonably small way. (Similar to Construction~\ref{const.usual_algebra} one may start with a tensor algebra and kill everything superfluous, however this typically gives results that are too big to control.)
\end{remark}

\section{Mutation} \label{sec.mutation}

In this section, we describe the mutation procedure on a graded quiver with differential. Let $Q$ be as in Section~\ref{sect.setup}, and $i \in \{1, \ldots, n\}$ a vertex of $Q$ without loops of degree $0$ attached to it. We construct a new quiver $M$. Since many of the arrows will have the same names as arrows of $Q$, we denote the degrees of these in $M$ by $\| \cdot \|$, and the differential by $\partial$.

The main result of this paper (Theorem~\ref{thm.main_result}), which we prove in Section~\ref{sec.proof}, is that the quiver $M$ described here is in fact the silting mutation of the quiver $Q$ we started with.

We first describe the quiver $M$. While the description of $M$ itself is fairly simple, the description of the differential is somewhat involved, and in particular it seems non-trivial that the formulas actually define a differential (i.e.\ that the candidate differential squares to $0$).

We denote by $A$ the collection of all arrows of degree $0$ which start in $i$. (Recall that we assume that none of these arrows is a loop.) For the arrows in $A$ we will use greek letters from the beginning of the alphabet (mostly $\alpha$, sometimes $\beta$).

Since we often have to distinguish according to whether a given vertex is $i$ or not, we denote by $\ci{\cdot}$ the characteristic function of $i$, that is
\[ \ci{j} = \begin{cases} 1 & \text{if } j = i \\ 0 & \text{if } j \neq i \end{cases}. \]

\subsection*{Step 1: Rotation of arrows}

We increase the degree of all arrows into $i$, and decrease the degree of all arrows out of $i$. In particular all loops in $i$ keep their degree. Note that now precisely the arrows in $A$ have been assigned degree $-1$. We replace these arrows $\alpha$ by arrows $\alpha^*$ of degree $0$ in the opposite direction.

In short terms, we have arrows:
\begin{align*}
& \varphi \colon j \to[.6] j' && \text{for } \varphi \colon j \to[.6] j' \in Q_1 \setminus A && \| \varphi \| = |\varphi| - \ci{\st{\varphi}} + \ci{\ta{\varphi}}; \\
& \alpha^* \colon j \to[.6] i && \text{for } \alpha \colon i \to[.6] j \in A && \| \alpha^* \| = 0. \\
\end{align*}

\subsection*{Step 2: Composition arrows}

As the reader familiar with cluster theory might expect, we need to add new arrows for ``broken compositions'', that is we introduce arrows $\alpha \varphi$ for any $\varphi$ ending in $i$ and $\alpha \in A$.

That is, we have
\begin{align*}
 & \alpha \varphi \colon j \to[.6] j' && \text{for } \varphi \colon j \to[.6] i \in Q_1, \alpha \colon i \to[.6] j' \in A && \| \alpha \varphi \| = \| \varphi \| - 1. \\
\end{align*}

Note that using the name ``$\alpha \varphi$'' is not ambiguous, since $\alpha$ does not appear on its own in the mutated quiver.

\subsection*{Step 3: Anti-compositions}

Symmetric to the above, we also need to deal with ``new'' compositions, that only came into existence because we turned arrows. This is done by introducing the following arrows.
\begin{align*}
& \varphi \alpha^{-1} \colon j \to[.6] j' && \text{for } \varphi \colon i \to[.6] j'  \in Q_1 \setminus A, \alpha \colon i \to[.6] j \in A && \| \varphi \alpha^{-1} \| = \| \varphi \| + 1 \\ 
\end{align*}

\subsection*{Step 4: In case of cycles at $i$}

In case there are cycles in $i$ we need to combine the two constructions from Steps~3 and 4 above introduce the arrows
\begin{align*}
& \alpha \varphi \beta^{-1} \colon j \to[.6] j' && \text{for } \varphi \colon i \to[.6] i \in Q_1, \alpha \colon i \to[.6] j', \beta \colon i \to[.6] j \in A && \| \alpha \varphi \beta^{-1} \| = \| \varphi \|.
\end{align*}

In other words, paths in $M$ are just composable sequences of arrows of $Q$, plus arrows $\alpha^*$ and $\alpha^{-1}$ for $\alpha \in A$, subject to the rules that
\begin{itemize}
\item an arrow $\alpha \in A$ must always be preceded by an arrow $\varphi \in Q_1$;
\item an arrow $\alpha^{-1}$ must always be followed by an arrow $\varphi \in Q_i \setminus A$.  
\end{itemize}

\subsection*{The differential}

We denote the differential for $M$ by $\partial$. To define it, we will need the following constructions on paths in $M$.

\begin{construction} \label{const.red_dec}
The \emph{reduction} of a (linear combination of) path(s), denoted by $\red(p)$, is obtained by just removing all paths that start in an arrow $\alpha \in A$ (i.e.\ where this arrow is the rightmost one). Thus $\red$ defines a (non-degree preserving) map from $kQ$ to $kM$.

In the other extreme, we denote by $p / \alpha$ the linear combination obtained by just remembering paths that start with $\alpha$, and then removing the $\alpha$.

That is
\[ p = \red(p) + \sum_{\alpha \in A} p/\alpha \; \alpha. \]

\medskip
The \emph{decorated} version of a path $p$, denoted by $\dec(p)$, is obtained by introducing the symbol $\Delta = 1 - (\sum_{\alpha \in A} \alpha^{-1} \alpha)$ wherever this is allowed - that is whenever the path passes through vertex $i$, coming in through an arrow in $Q_1$ and leaving through an arrow in $Q_1 \setminus A$.
\end{construction}

\begin{definition} \label{def.new_diff}
We define $\partial$ by
\begin{align*}
\partial(\varphi) & = \begin{cases} \dec( \red( d \varphi)) & \text{ if } \ta{\varphi} \neq i \\ \sum_{\alpha \in A} \alpha^* \alpha \varphi - \dec( \red( d \varphi)) & \text{ if } \ta{\varphi} = i \end{cases} \\
\partial(\alpha \varphi) & = \dec( \alpha \red ( d \varphi)) \\
\partial(\alpha^*) & = 0 \\
\partial( \varphi \alpha^{-1}) & = \partial(\varphi) \alpha^{-1} + (-1)^{\| \varphi \|} \varphi \alpha^*  - (-1)^{\ci{ \ta{\varphi}}} \dec( d\varphi / \alpha) \\
\partial(\alpha \varphi \beta^{-1}) & = \partial(\alpha \varphi) \beta^{-1} + (-1)^{\| \alpha \varphi \|} \alpha \varphi \beta^* - \alpha \dec( d \varphi / \beta )
\end{align*}
The map $\partial$ is extended to (linear combinations of) paths using linearity and the Leibnitz rule.
\end{definition}

It is not clear a priori that this defines a differential, i.e.\ that $\partial^2 = 0$. We check this in Section~\ref{sec.is_diff}.

\begin{theorem}[{See Theorem~\ref{thm.is_diff}}]
$\partial^2 = 0$.
\end{theorem}

\section{Examples} \label{sec.examples}

\begin{example} \label{example.APR_tilt}
Let $A = k[1 \to^{\varphi} 2 \to^{\psi} 3] / (\psi \varphi)$ as in Example~\ref{ex.A_3/rad^2}. Then we have seen that $A$ is quasi-isomorphic to the graded quiver with differential given by
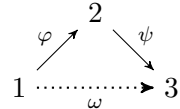
\[ \begin{tikzpicture}[baseline=-3.5pt]
 \node (1) at (0,0) {$1$};
 \node (2) at (1,1) {$2$};
 \node (3) at (2,0) {$3$};
 \draw [->] (1) to node [above left=-2pt] {$\scriptstyle \varphi$} (2);
 \draw [->] (2) to node [above right=-2pt] {$\scriptstyle \psi$} (3);
 \draw [->, dotted, thick] (1) to node [below] {$\scriptstyle \omega$} (3); 
\end{tikzpicture} \]
where the dotted arrow $\omega$ is in degree $1$ and $d \omega = \psi \varphi$.

We silting (in fact tilting) mutate in $1$. Here $A = \{ \varphi\}$. According to our rule obtain the new quiver is
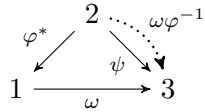
\[ \begin{tikzpicture}[baseline=-3.5pt]
 \node (1) at (0,0) {$1$};
 \node (2) at (1,1) {$2$};
 \node (3) at (2,0) {$3$};
 \draw [->] (2) to node [above left=-2pt] {$\scriptstyle \varphi^*$} (1);
 \draw [->] (2) to node [below left=-2pt] {$\scriptstyle \psi$} (3);
 \draw [->] (1) to node [below] {$\scriptstyle \omega$} (3);
 \draw [->, dotted, thick, bend left=30] (2) to node [above right=-2pt] {$\scriptstyle \omega \varphi^{-1}$} (3);
\end{tikzpicture} \]
with
\[ \partial( \omega \varphi^{-1} ) = \omega \varphi^* - \underbrace{\dec( \underbrace{d \omega}_{= \psi \varphi} / \varphi )}_{= \psi} = \omega \varphi^* - \psi. \]
Factoring out the dg-ideal $(\omega \varphi^{-1}, \partial(\omega \varphi^{-1}) )$ we obtain the quasi-isomorphic quiver
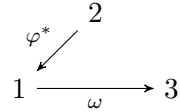
\[ \begin{tikzpicture}[baseline=-3.5pt]
 \node (1) at (0,0) {$1$};
 \node (2) at (1,1) {$2$};
 \node (3) at (2,0) {$3$};
 \draw [->] (2) to node [above left=-2pt] {$\scriptstyle \varphi^*$} (1);
 \draw [->] (1) to node [below] {$\scriptstyle \omega$} (3);
\end{tikzpicture} \]
(as the reader familiar with tilted algebras of type $A_3$ will have expected).
\end{example}

\begin{example}
Consider the quiver with relations
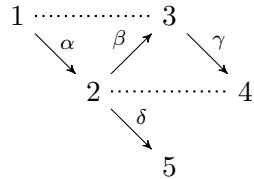
\[ \begin{tikzpicture}
 \node (1) at (0,2) {$1$};
 \node (2) at (1,1) {$2$};
 \node (3) at (2,2) {$3$};
 \node (4) at (3,1) {$4$};
 \node (5) at (2,0) {$5$};
 \draw (1) [->] to node [above right=-2pt] {\scriptsize $\alpha$} (2);
 \draw (2) [->] to node [above left=-2pt] {\scriptsize $\beta$} (3);
 \draw (3) [->] to node [above right=-2pt] {\scriptsize $\gamma$} (4);
 \draw (2) [->] to node [above right=-2pt] {\scriptsize $\delta$} (5);
 \draw (1) [dotted, thick] to (3);
 \draw (2) [dotted, thick] to (4);
\end{tikzpicture} \]

This is quasi-isomorphic to the dg quiver
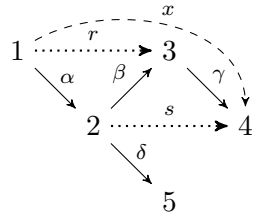
\[ \begin{tikzpicture}
 \node (1) at (0,2) {$1$};
 \node (2) at (1,1) {$2$};
 \node (3) at (2,2) {$3$};
 \node (4) at (3,1) {$4$};
 \node (5) at (2,0) {$5$};
 \draw (1) [->] to node [above right=-2pt] {\scriptsize $\alpha$} (2);
 \draw (2) [->] to node [above left=-2pt] {\scriptsize $\beta$} (3);
 \draw (3) [->] to node [above right=-2pt] {\scriptsize $\gamma$} (4);
 \draw (2) [->] to node [above right=-2pt] {\scriptsize $\delta$} (5);
 \draw (1) [->,dotted, thick] to node [above] {\scriptsize $r$} (3);
 \draw (2) [->,dotted, thick] to node [above] {\scriptsize $s$} (4);
 \draw (1) [->,dashed,out=30,in=90] to node [above] {\scriptsize $x$} (4);
\end{tikzpicture} \]
with dotted arrows in degree $1$, the dashed arrow in degree $2$, and
\[ d(r) = \beta \alpha, \qquad d(s) = \gamma \beta, \qquad d(x) = s \alpha - \gamma r. \]
Mutating at vertex $2$, we have $A = \{ \beta, \delta \}$. Applying our mutation rule we obtain
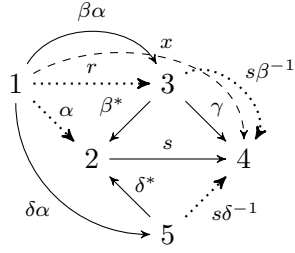
\[ \begin{tikzpicture}
 \node (1) at (0,2) {$1$};
 \node (2) at (1,1) {$2$};
 \node (3) at (2,2) {$3$};
 \node (4) at (3,1) {$4$};
 \node (5) at (2,0) {$5$};
 \draw (1) [->,dotted,thick] to node [above right=-2pt] {\scriptsize $\alpha$} (2);
 \draw (3) [->] to node [above left=-2pt] {\scriptsize $\beta^*$} (2);
 \draw (3) [->] to node [above right=-2pt] {\scriptsize $\gamma$} (4);
 \draw (5) [->] to node [above right=-2pt] {\scriptsize $\delta^*$} (2);
 \draw (1) [->,dotted, thick] to node [above] {\scriptsize $r$} (3);
 \draw (2) [->] to node [above] {\scriptsize $s$} (4);
 \draw (1) [->,dashed,out=30,in=90] to node [above] {\scriptsize $x$} (4);
 \draw (1) [->,out=60,in=120] to node [above] {\scriptsize $\beta \alpha$} (3);
 \draw (1) [->,out=-90,in=180] to node [below left=-2pt] {\scriptsize $\delta \alpha$} (5);
 \draw (5) [->,dotted,thick] to node [below right=-2pt] {\scriptsize $s \delta^{-1}$} (4);
 \draw (3) [->,dotted,thick,out=30,in=60] to node [above right=-2pt] {\scriptsize $s \beta^{-1}$} (4);
\end{tikzpicture} \]
with differential given by
\begin{align*}
\partial( \alpha ) & = \beta^* \beta \alpha + \delta^* \delta \alpha \\
\partial(r) & = \beta \alpha \\
\partial(s) & = 0 \\
\partial(x) & = s \Delta \alpha - \gamma r = s \alpha - s \beta^{-1} \beta \alpha - s \delta^{-1} \delta \alpha - \gamma r \\
\partial(s \beta^{-1}) & = s \beta^* - \dec( ds / \beta) = s \beta^* - \gamma \\
\partial(s \delta^{-1}) & = s \delta^*
\end{align*}
As in Example~\ref{example.APR_tilt} above, we may divide out the ideals $(r, \partial(r))$ and $(s \beta^{-1}, \partial(s \beta^{-1}))$ without changing the algebra up to quasi-isomorphism. Thus we are left with
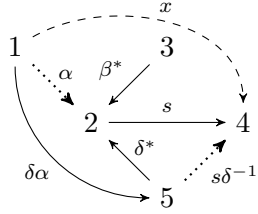
\[ \begin{tikzpicture}
 \node (1) at (0,2) {$1$};
 \node (2) at (1,1) {$2$};
 \node (3) at (2,2) {$3$};
 \node (4) at (3,1) {$4$};
 \node (5) at (2,0) {$5$};
 \draw (1) [->,dotted,thick] to node [above right=-2pt] {\scriptsize $\alpha$} (2);
 \draw (3) [->] to node [above left=-2pt] {\scriptsize $\beta^*$} (2);
 \draw (5) [->] to node [above right=-2pt] {\scriptsize $\delta^*$} (2);
 \draw (2) [->] to node [above] {\scriptsize $s$} (4);
 \draw (1) [->,dashed,out=30,in=90] to node [above] {\scriptsize $x$} (4);
 \draw (1) [->,out=-90,in=180] to node [below left=-2pt] {\scriptsize $\delta \alpha$} (5);
 \draw (5) [->,dotted,thick] to node [below right=-2pt] {\scriptsize $s \delta^{-1}$} (4);
\end{tikzpicture} \]
with
\[ \partial(x) = s \alpha - s \delta^{-1} \delta \alpha, \qquad \partial(\alpha) = \delta^* \delta \alpha, \qquad \partial( s \delta^{-1}) = s \delta^*. \]
(We may observe that this is, in fact, quasi-isomorphic to $A^{\operatorname{op}}$.)
\end{example}

\section{Ginzburg algebras} \label{sec.ginzburg}

In this section we apply our mutation rule to Ginzburg algebras, reproving a result of Keller and Yang \cite{keller-yang} which asserts that the silting mutation of Ginzburg algebras coincides with the Ginzburg algebra of the mutated quiver.

While all results in this section are contained in the corresponding results of the next section, we feel it is worth discussing the more classical case ($d=3$, in terms of the next section) here separately. Firstly, because the usual notation differs somewhat (here we have a clear distinction between ``original" and ``opposite" arrows, while the difference disappears in the next section), and secondly because some technicalities that appear in the general case do not appear here.

We start by recalling the relevant notation.
\begin{definition}[See \cite{derksen-weyman-zelevinsky}]
Let $Q$ be a quiver.

A \emph{potential} on $Q$ is an element $W \in \operatorname{HH}_0(kQ) = kQ / [kQ, kQ]$, which we may think of as a linear combination of cycles up to cyclic permutation.

For an arrow $\alpha$ of $Q$, the \emph{cyclic derivative} of a cycle $c$ is given as $\partial_{\alpha} c = \sum_{c = p \alpha q} qp$. Clearly this extends to linear combinations of cycles.

For a vertex $i$ of $Q$ without loops attached to it, the mutation $\mu_i(Q, W)$ is given by
\begin{itemize}
\item the quiver $\mu_i(Q)$ obtained by reversing all arrows starting or ending in $i$ (these new reversed arrows will be denoted by $\alpha^*$ if the original arrow was called $\alpha$) and adding compositions $[\beta \alpha]$ for any composition of arrows passing through the vertex $i$;
\item the potential
\[ \mu_i(W) = W + \sum_{\alpha \colon x \smash{\to} i} \sum_{\beta \colon i \smash{\to} y} [\beta \alpha] \alpha^* \beta^*. \]
\end{itemize}
\end{definition}

\begin{definition}[See \cite{ginzburg}]
The \emph{Ginzburg algebra} associated to a quiver with potential $(Q, W)$ is the dg-algebra given as the quiver
\begin{align*}
Q^{\rm Gin} & = Q \cup \{ \alpha^{\op} \colon j \to i \mid \alpha \colon i \to j \text{ in } Q \} \\
& \qquad \cup \{ \ell_i \colon i \to i \mid i \text{ a vertex of } Q \}
\end{align*}
where the arrows $\alpha^{\op}$ are in degree $1$ and the loops $\ell_i$ are in degree $2$.

The differential is given as
\begin{align*}
 d \alpha^{\op} & = \partial_{\alpha} W \qquad \text{and} \\
 d \ell_i & = \sum_{\alpha \colon i \smash{\to} x} \alpha^{\op} \alpha - \sum_{\alpha \colon x \smash{\to} i} \alpha \alpha^{\op}.
\end{align*}
\end{definition}

Now we can state the main result of this section, which was originally proved by Keller and Yang.

\begin{theorem}
Let $(Q, W)$ be a quiver with potential. Then the derived endomorphism ring of the silting mutation of $(kQ^{\rm Gin}, d)$ in $i$ is the Ginzburg algebra of $\mu_i(Q, W)$.
\end{theorem}

\begin{proof}
Clearly we are in the setup of Theorem~\ref{thm.main_result}, so we just apply it formally.

The mutated quiver $M$ as described in Section~\ref{sec.mutation} has arrows as follows:

\noindent
Step 1: Rotation of arrows
\begin{itemize}
 \item all arrows not involving $i$ just like $Q^{\rm Gin}$;
 \item for every arrow $\varphi \colon x \to i \in Q$ an arrow $\varphi \colon x \to i$ of degree $1$ and an arrow $\varphi^{\op} \colon i \to x$ of degree $0$;
 \item for every arrow $\alpha \colon i \to x \in Q$ an arrow $\alpha^* \colon x \to i$ of degree $0$ and an arrow $\alpha^{\op} \colon x \to i$ of degree $2$;
 \item the loop $\ell_i$ of degree $2$;
\end{itemize}
Step 2: Compositions
\begin{itemize}
\item for every composition $x \to^{\varphi} i \to^{\alpha} y$ in $Q$ an arrow $\alpha \varphi \colon x \to y$ of degree $0$;
 \item for every subquiver $x \arrow{<-}^{\beta} i \to^{\alpha} y$ in $Q$ (where $\alpha$ and $\beta$ might coincide), an arrow $\alpha \beta^{\op} \colon x \to y$ of degree $1$;
 \item for every arrow $\alpha \colon i \to x$ in $Q$ a composition arrow $\alpha \ell_i$;
\end{itemize}
Step 3: Anti-compositions
\begin{itemize}
 \item for every composition $x \to^{\varphi} i \to^{\alpha} y$ in $Q$ an arrow $\varphi^{\op} \alpha^{-1}$ of degree $1$;
 \item for every arrow $\alpha \colon i \to x$ in $Q$ an arrow $\ell_i \alpha^{-1} \colon x \to i$ of degree $3$;
\end{itemize}
Step 4
\begin{itemize}
 \item for every subquiver $x \arrow{<-}^{\beta} i \to^{\alpha} y$ in $Q$ (where $\alpha$ and $\beta$ might coincide), an arrow $\alpha \ell_i \beta^{-1} \colon x \to y$ of degree $2$.
\end{itemize}
Since these are a lot of arrows, and might seem confusing, let us depict the situation in case of the quiver
\[ Q = 1 \to^{\varphi} 2 \to^{\alpha} 3 \]
for $i = 2$. This quiver gives a good impression of the general situation. We have
\[ Q^{\rm Gin} =
\begin{tikzpicture}[baseline=-3pt,inner sep=1pt]
 \node (1) at (0,0) {$1$};
 \node (2) at (2,0) {$2$};
 \node (3) at (4,0) {$3$};
 \draw [->,bend left=20] (1) to node [above] {$\scriptstyle \varphi$} (2);
 \draw [->,bend left=20] (2) to node [pos=.1,sloped,rotate=90] {$-$} node [below] {$\scriptstyle \varphi^{\op}$} (1);
 \draw [->,bend left=20] (2) to node [above] {$\scriptstyle \alpha$} (3);
 \draw [->,bend left=20] (3) to node [pos=.1,sloped,rotate=90] {$-$} node [below] {$\scriptstyle \alpha^{\op}$} (2);
 \draw [->,out=60,in=120,looseness=12] (1) to node [pos=.1,sloped,rotate=90] {$=$} node [above] {$\scriptstyle \ell_1$} (1);
 \draw [->,out=60,in=120,looseness=12] (2) to node [pos=.1,sloped,rotate=90] {$=$} node [above] {$\scriptstyle \ell_2$} (2);
 \draw [->,out=60,in=120,looseness=12] (3) to node [pos=.1,sloped,rotate=90] {$=$} node [above] {$\scriptstyle \ell_3$} (3);
\end{tikzpicture} \]
where the tags denote the degree of the arrows. The mutated quiver is
\[ M =
\begin{tikzpicture}[baseline=-3pt,inner sep=1pt]
 \node (1) at (0,0) {$1$};
 \node (2) at (2,0) {$2$};
 \node (3) at (4,0) {$3$};
 \draw [->,bend left=20] (1) to node [pos=.1,sloped,rotate=90] {$-$} node [above] {$\scriptstyle \varphi$} (2);
 \draw [->,bend left=20] (2) to node [below] {$\scriptstyle \varphi^{\op}$} (1);
 \draw [->,bend right=20] (3) to node [above] {$\scriptstyle \alpha^*$} (2);
 \draw [->,bend left=20] (3) to node [pos=.1,sloped,rotate=90] {$=$} node [below] {$\scriptstyle \alpha^{\op}$} (2);
 \draw [->,out=60,in=120,looseness=12] (1) to node [pos=.1,sloped,rotate=90] {$=$} node [above] {$\scriptstyle \ell_1$} (1);
 \draw [->,out=60,in=120,looseness=12] (2) to node [pos=.1,sloped,rotate=90] {$=$} node [above] {$\scriptstyle \ell_2$} (2);
 \draw [->,out=60,in=120,looseness=12] (3) to node [pos=.1,sloped,rotate=90] {$=$} node [above] {$\scriptstyle \ell_3$} (3);
 \draw [->,bend left=40,looseness=1.5] (3) to node [pos=.1,sloped,rotate=90] {$\equiv$} node [below] {$\scriptstyle \ell_2 \alpha^{-1}$} (2);
 \draw [->,bend left=40,looseness=1.5] (2) to node [pos=.1,sloped,rotate=90] {$-$} node [above] {$\scriptstyle \alpha \ell_2$} (3);
 \draw [->,out=-20,in=20,looseness=50] (3) to node [pos=.1,sloped,rotate=90] {$-$} node [left] {$\scriptstyle \alpha \alpha^{\op}$} (3);
 \draw [->,out=-40,in=40,looseness=30] (3) to node [pos=.1,sloped,rotate=90] {$=$} node [right] {$\scriptstyle \alpha \ell_2 \alpha^{-1}$} (3);
 \draw [->,bend right=85] (1) to node [above] {$\scriptstyle \alpha \varphi$} (3);
 \draw [->,bend left=95,looseness=1.2] (3) to node [pos=.1,sloped,rotate=90] {$-$} node [below] {$\scriptstyle \varphi^{\op} \alpha^{-1}$} (1);
\end{tikzpicture} \]
The reader will note that this contains more arrows than the Ginzburg algebra of the mutated quiver, and in fact it even contains an arrow of degree $3$, which does not occur in Ginzburg algebras.

We observe
\[ \partial( \ell_i \alpha^{-1} ) = \partial(\ell_i) \alpha^{-1} + \ell_i \alpha^* + \underbrace{(d \ell_i) / \alpha}_{= \alpha^{\op}} . \]
Therefore, up to quasi-isomorphism, the arrows $\ell_i \alpha^{-1}$ and $\alpha^{\op}$ cancel.

Similarly the arrows $\alpha \ell_i \beta^{-1}$ and $\alpha \beta^{\op}$ cancel.

Thus we are left with
\begin{itemize}
 \item all arrows not involving $i$, plus $\ell_i$, exactly as in $Q^{\rm Gin}$;
 \item for an incoming arrow $\varphi \colon x \to i$ in $Q$: arrows $\varphi \colon x \to i$ and $\varphi^{\op} \colon i \to x$ of degrees $1$ and $0$, respectively;
 \item for an outgoing arrow $\alpha \colon i \to x$ in $Q$: arrows $\alpha^* \colon x \to i$ and $\alpha \ell_i \colon i \to x$ of degrees $0$ and $1$, respectively;
 \item for any sequence of arrows $x \to^{\varphi} i \to^{\alpha} y$ in $Q$: arrows $\alpha \varphi \colon x \to y$ and $\varphi^{\op} \alpha^{-1} \colon y \to x$ of degrees $0$ and $1$, respectively.
\end{itemize}
This is exactly the quiver of the Ginzburg algebra of the mutated quiver. It only remains to check that also the differential $\partial$ on this new quiver coincides with the differential of the Ginzburg algebra.

Clearly for the arrows $\varphi^{\op}$ not involved in the vertex $i$ the differential remains what it was, and for all arrows of degree $0$ the differential will be $0$. We calculate:

For an arrow $\varphi$ of $Q$ ending in $i$, we have
\[ \partial(\varphi) = \sum_{\mathclap{\alpha \colon i \smash{\to} x}} \alpha^* \alpha \varphi; \]
For an arrow $\alpha$ of $Q$ starting in $i$, we have 
\[ \partial(\alpha \ell_i) = - \sum_{\varphi \colon x \smash{\to} i} \alpha \varphi \varphi^{\op}. \]
For an arrow of the form $\varphi^{\op} \alpha^{-1}$ we obtain
\[ \partial( \varphi^{\op} \alpha^{-1}) = \varphi^{\op} \alpha^* - \dec ((d \varphi^{\op}) / \alpha) = \varphi^{\op} \alpha^* - \partial_{\varphi}W / \alpha. \]
Note that the decoration can be omitted here, since $d \varphi^{\op}$ is of degree $0$, and thus can never pass through vertex $i$ and leave via an arrow of positive degree.

For the loop at the mutation vertex we obtain
\[ \partial(\ell_i) = \sum_{\mathclap{\alpha \colon i \smash{\to} x}} \alpha^* \alpha \ell_i + \sum_{\mathclap{\varphi \colon x \smash{\to} i}} \varphi \varphi^{\op}. \]
Finally, for any loop $\ell_j$ at a vertex $j \neq i$, we have
\begin{align*}
\partial(\ell_j) & = \dec ( \sum_{\mathclap{\varphi \colon j \smash{\to} x}} \varphi^{\op} \varphi - \sum_{\mathclap{\varphi \colon x \smash{\to} j}} \varphi \varphi^{\op} ) \\
& = \; \sum_{\mathclap{\substack{\alpha \colon j \smash{\to} x \\ x \neq i}}} \varphi^{\op} \varphi - \sum_{\mathclap{\substack{\varphi \colon x \smash{\to} j \\ x \neq i}}} \varphi \varphi^{\op} + \sum_{\mathclap{\varphi \colon j \smash{\to} i}} \varphi^{\op} \Delta \varphi - \sum_{\mathclap{\alpha \colon i \smash{\to} j}} \alpha \alpha^{\op} ) \\
& = \; \sum_{\mathclap{\alpha \colon j \smash{\to} x}} \varphi^{\op} \varphi - \sum_{\mathclap{\substack{\varphi \colon j \smash{\to} x \\ \alpha \colon i \smash{\to} y}}} \varphi^{\op} \alpha^{-1} \alpha \varphi - \sum_{\mathclap{\substack{\alpha \colon x \smash{\to} j \\ x \neq i}}} \varphi \varphi^{\op} - \sum_{\mathclap{\alpha \colon i \smash{\to} j}} \alpha \alpha^{\op}
\end{align*}
We need to recall how we canceled $\alpha \alpha^{\op}$: In fact we canceled the differential of $\alpha \ell_i \alpha^{-1}$, which is
\[ \partial( \alpha \ell_i \alpha^{-1}) = \partial( \alpha \ell_i ) \alpha^{-1} - \alpha \ell_i \alpha^* - \alpha \alpha^{\op} \]
whence we get that
\[ - \alpha \alpha^{\op} = \sum_{\varphi \colon x \smash{\to} i} \alpha \varphi \varphi^{\op} \alpha^{-1} + \alpha \ell_i \alpha^*. \]
Inserting above we obtain
\[ \partial(\ell_j) = \; \sum_{\mathclap{\alpha \colon j \smash{\to} x}} \varphi^{\op} \varphi - \sum_{\mathclap{\substack{\varphi \colon j \smash{\to} x \\ \alpha \colon i \smash{\to} y}}} \varphi^{\op} \alpha^{-1} \alpha \varphi - \sum_{\mathclap{\substack{\alpha \colon x \smash{\to} j \\ x \neq i}}} \varphi \varphi^{\op} + \sum_{\mathclap{\substack{\alpha \colon i \smash{\to} j \\ \varphi \colon x \smash{\to} i}}} \alpha \varphi \varphi^{\op} \alpha^{-1} + \sum_{\mathclap{\alpha \colon i \smash{\to} j}} \alpha \ell_i \alpha^*. \]
Comparing, one sees that we do have the differential for the Ginzburg algebra of the mutated quiver. (Except for a few signs, which can be fixed by replacing all arrows $\varphi^{\op} \colon i \to x$, $\varphi^{\op} \alpha^{-1}$, and the loop $\ell_i$ by their negative.)
\end{proof}

\section{Higher Ginzburg algebras} \label{sec.higher_ginzburg} \label{sec.higher_ginz}

In this section we apply our mutation rule to higher Ginzburg algebras. In particular we will observe that the mutation of a higher Ginzburg algebra is again a higher Ginzburg algebra. Throughout, we fix a number $d \geq 2$, which will be the Calabi-Yau dimension of our Ginzburg algebra.

We start by recalling the relevant notation.

\begin{definition}[See \cite{van_den_bergh}] \label{def.higher_Ginzburg}
Let $Q$ be a graded quiver, with arrows in degrees $0$, \dots, $d-2$. For any arrow $\varphi$ of $Q$, there is an arrow (up to sign) $\varphi^{\op}$, in the opposite direction, and such that $| \varphi| + |\varphi^{\op}| = d - 2$. Here ``up to sign'' means that we allow $\varphi^{\op}$ to be either an arrow itself, or the negative of an arrow. Moreover we require that $\varphi^{\op \op} = - (-1)^{|\varphi| | \varphi^{op}|} \varphi$, that is up to signs the opposite of the opposite is the original arrow.

Assume $W$ is a \emph{potential} on $Q$, that is a linear combination of cycles up to (signed) cyclic permutation, which is homogeneous of degree $d-3$. Moreover the potential is required to satisfy the condition that the necklace bracket $\{ W, W \}$ vanishes. For us the meaning of this condition is that the differential on the Ginzburg algebra defined below squares to $0$, so we ignore the technicalities and assume this.

For an arrow $\varphi$ of $Q$, the \emph{cyclic derivative} of a cycle $c$ is given as $\partial_{\alpha} c = \sum_{c = p \alpha q} (-1)^{|p \alpha| |q |} qp$. Clearly this extends to linear combinations of cycles.

The \emph{$d$-dimensional Ginzburg algebra} associated to $(Q, W)$ is given by the quiver $\bar{Q}$ obtained from $Q$ by adding a loop $\ell_j$ of degree $d-1$ at every vertex $j$. The differential is given as follows:
\begin{align*}
d \varphi & = \partial_{\varphi^{\rm op}} W && \varphi \text{ an arrow in } Q, \\
d \ell_i & = \sum_{\varphi} \varphi \varphi^{\rm op} && i \text{ a vertex of } Q.
\end{align*}
where the sum runs over all arrows $\varphi$ of $Q$ ending in vertex $i$.
\end{definition}

\begin{remark} \label{rem.orig=opp}
Comparing to Van den Bergh's conventions in \cite[Section~10.3]{van_den_bergh}, we note the following
\begin{itemize}
\item Van den Bergh starts with a quiver, and then adds new opposite arrows for all the arrows of the original quiver. In our terminology here we start with a quiver which already contains all these arrows.
\item Van den Bergh ``derives on the left'', while we ``derive on the right'' -- therefore we obtain the sign $(-1)^{|p \alpha| |q |}$ instead of his  $(-1)^{|p| | \alpha q |}$.
\item In Van den Bergh's notation there are different signs in the definition of the differential of original and opposite arrows. Here we compensate for those by the sign that appears when taking the opposite twice. 
\end{itemize}
\end{remark}

\subsection*{Mutation}

We have the following notion of mutation of higher dimensional quivers with potential:

\begin{definition}
Let $(Q, W)$ be as in Definition~\ref{def.higher_Ginzburg} above. Let $i$ be a vertex of $Q$ without any loops of degree $0$ attached to it. Then the \emph{right mutation} $\mu^{\rm R}_i(Q, W)$ is defined by the following construction.

First we construct a new graded quiver $Q\m$:
\begin{enumerate}
\item increase by $1$ the degree of all arrows ending in vertex $i$, and decrease the degree of all arrows starting in $i$;
\item for any arrow $\alpha$ of degree $0$ starting in $i$ in the original quiver (i.e.\ any arrow that was assigned degree $-1$ in step 1): remove $\alpha$ and $\alpha^{\op}$, and introduce a new pair of arrows $\alpha^*$ and $(\alpha^*)^{\op\m}$ in the opposite direction;
\end{enumerate}
(We may note that we could have done these two things in one step, by just saying we calculate modulo $d-1$ in (1). However in the sequel it will be an advantage to not risk confusing the arrows $\alpha$ and $(\alpha^*)^{\op}$.)

It should also be noted that we might have destroyed our assumption $\varphi^{\op \op} = - (-1)^{| \varphi| | \varphi^{\op} |} \varphi$ when changing the degrees -- this forces us to change some signs in the ``$\op$", at least if $d$ is even. We choose to set $\varphi^{\op\m} = (-1)^{(d+1) \ci{ \rm{t}_{\varphi}}} \varphi^{\op}$, that is we change the sign of the opposite of all arrows ending in $i$, provided $d$ is even.
\begin{enumerate}
\setcounter{enumi}{2}
\item add new arrows for all ``broken compositions'': for any arrow $\varphi$ ending in $i$ of degree at most $d-3$, and any arrow $\alpha$ of degree $0$ leaving $i$, add a new arrow which is the formal composition $\alpha \varphi$, and it's opposite $(\alpha \varphi)^{\op\m} = - \varphi^{\op} \alpha^{-1}$; (As before we note that this notation should not create any ambiguity, since $\alpha$ and $\alpha^{-1}$ only appear in such composition arrows, and not as arrows of their own right.)
\item for any loop $\varphi$ at $i$ in $Q$ (i.e.\ the additional loop $\ell_i$ does not count), and any two arrows $\alpha$ and $\beta$ of degree $0$ starting in $i$, we introduce a new formal composite arrow $\alpha \varphi \beta^{-1}$ of degree identical to the degree of $\varphi$. (This arrow's opposite will be given as $(\alpha \varphi \beta^{-1})^{\op\m} = \beta \varphi^{\op} \alpha^{-1}$.
\end{enumerate}
The potential on $Q\m$ is given as
\[ W_{\rm M} = \dec_{\rm cyc} W + \sum_{\alpha, \varphi} \alpha \dec( \varphi \varphi^{\op\m} ) \alpha^*, \]
where the sum runs over all arrows $\varphi$ of $Q$ ending in $i$, and all arrows $\alpha$ of $Q$ of degree $0$ starting in $i$.

Here $\dec_{\rm cyc} W$ is to be understood cyclically, that is we also allow introduction of $\Delta$ ``between the ends of the cycle''. Explicitly, for a cycle $c$ this means
\[ \dec_{\rm cyc} c = \begin{cases} \dec c & \text{if $c$ does not start and end in $i$} \\ (-1)^d \dec c - \sum \alpha (\dec c) \alpha^{-1} & \text{if $c$ starts and ends in $i$,} \end{cases}  \]
where we assume that $c$ does not start with an arrow $\alpha \in A$ (this may always be achieved by cyclic permutation).
See Remark~\ref{rem.dec_cyclic} below for an explanation of the sign.
\end{definition}

\begin{remark} \label{rem.dec_cyclic}
To motivate the sign appearing in the definition of cyclic decoration, consider a cycle $c$ starting and ending in $i$, but passing through at least one different vertex. Then we have
\begin{align*}
c & = pq \qquad \qquad \qquad \text{such that } \ta{q} = \st{p} \neq i \\
& \sim (-1)^{|p| |q|} qp \qquad \qquad \text{(signed cyclic equivalence)} \\
\overset{\dec}{\smash{\mapsto}\vphantom{=}} & \phantom{=} (-1)^{|p| |q|} \dec(q) \Delta \dec(p) \\
& =  (-1)^{|p| |q|} \dec(q) \dec(p) \\ & - \sum_{\alpha} (-1)^{|p| |q|} \dec(q) \alpha^{-1} \cdot \alpha \dec(p) \\
& \sim  (-1)^{|p| |q|} (-1)^{\| \dec q \| \| \dec p \|} \dec(p) \dec(q) \\ & -  \sum_{\alpha} (-1)^{|p| |q|} (-1)^{\| \dec q \alpha^{-1} \| \| \alpha \dec p \|} \alpha \dec(p) \dec(q) \alpha^{-1} \\
& = (-1)^d \dec(c) - \sum_{\alpha} \alpha \dec(c) \alpha^{-1}
\end{align*}
where, for the last equality, we note that
\[ (-1)^{|p| |q|} (-1)^{\| \dec q \| \| \dec p \|} = (-1)^{|p| |q|} (-1)^{(|q|+1)(|p|-1)} = (-1)^{|p| + |q| + 1} = (-1)^d \]
if $c$ is of degree $d-3$, while $\| \dec(q) \alpha^{-1} \| = |q|$ and $\| \alpha \dec(p) \| = |p|$.
\end{remark}

In the following lemma, we calculate the cyclic derivatives of the mutated quiver with potential, in terms of the cyclic derivatives of the original.

\begin{lemma} \label{lem.der_of_mutated}
In the situation above, the cyclic derivatives of the mutated potential are
\begin{align*}
\partial_{\varphi^{\op\m}} W\m & = \dec \red \partial_{\varphi^{\op}} W \\ & \text{ for } \varphi \in Q \cap Q\m \text{ not ending in vertex $i$} \\
\partial_{\varphi^{\op\m}} W\m & = - \dec \partial_{\varphi^{\op}} W + \sum_{\alpha} \alpha^* (\alpha \varphi) \\ & \text{ for } \varphi \in Q \cap Q\m \text{ ending in vertex $i$} \\
\partial_{(\alpha^*)^{\op\m}} W\m & = 0 \\
\partial_{\alpha^*} W\m & = \sum_{\varphi} \alpha \dec( \varphi \varphi^{\op\m} ) \\
\partial_{(\alpha \varphi)^{\op\m}} W\m & = \alpha \dec \red (\partial_{\varphi^{\op}} W) \\
\partial_{(\varphi \alpha^{-1})^{\op\m}} W\m & = (\partial_{\varphi^{\op\m}} W\m ) \alpha^{-1} + (-1)^{\| \varphi \|} \varphi \alpha^* - (-1)^{\ci{\ta\varphi}} \dec (\partial_{\varphi^{\op}} W / \alpha) \\
\partial_{(\alpha \varphi \beta^{-1})^{\op\m}} W\m & = (\partial_{(\alpha \varphi)^{\op\m}} W\m ) \beta^{-1} + (-1)^{ \| \alpha \varphi \|} \alpha \varphi \beta^* - \alpha \dec (\partial_{\varphi^{\op}} W / \beta)
\end{align*}
\end{lemma}

\begin{proof}
For the first statement we note that $\varphi^{\op\m}$ only appear in $W\m$ in places where it already appeared in $W$. (Also note that since $\varphi$ does not end in $i$ we have $\varphi^{\op\m} = \varphi^{\op}$.) Finally note that when $\varphi^{\op\m}$ is directly followed by an arrow $\alpha \in A$, then the combination $(\alpha \varphi^{\op})$ will be considered a single arrow in the new quiver, and thus not contribute to $\partial_{\varphi^{\op\m}} W_M$ -- whence the ``$\red$" in the formula. Finally we need to decorate, because $W\m$ is obtained from $W$ by decorating.

For the second statement, since $\varphi$ now ends in $i$ we have $\varphi^{\op\m} = (-1)^{d+1} \varphi^{\op}$. Moreover, as demonstrated in Remark~\ref{rem.dec_cyclic}, an additional sign $(-1)^d$ appears when cyclically decorating a cycle starting in vertex $i$. Thus
\[ \partial_{\varphi^{\op\m}} \dec_{\rm cyc} W = - \dec \partial_{\varphi^{\op}} W. \]
To complete the proof of the second equation, note that
\[ \partial_{\varphi^{\op\m}} \sum_{\alpha, \psi} \alpha \dec( \psi \psi^{\op\m} ) \alpha^* = \sum_{\alpha} \alpha^* (\alpha \varphi). \]

The third statement is clear, since $(\alpha^*)^{\op\m}$ does not appear in $W\m$.

For the fourth statement, we first note that $\alpha^*$ does not appear in $\dec W$. That gives us the statement immediately.

When considering the fifth statement, recall first that we defined $(\alpha \varphi)^{\op\m} = - \varphi^{\op} \alpha^{-1}$. We note that the arrow $\varphi^{\op} \alpha^{-1}$ does not appear in the second summand of $W\m$, so
\begin{align*}
\partial_{(\alpha \varphi)^{\op\m}} W_M & = - \partial_{\varphi^{\op} \alpha^{-1}} \dec_{\rm cyc} W \\
& = \alpha \dec \red (\partial_{\varphi^{\op}} W).
\end{align*}

The last two statements are (maybe not surprisingly, considering the formulas) most involved to check. First observe that since $(\alpha \varphi^{\op})^{\op\m} = - \varphi^{\op \op} \alpha^{-1} = (-1)^{| \varphi | | \varphi ^{\op} | } \varphi \alpha^{-1}$ we have
\begin{align*}
(\varphi \alpha^{-1})^{\op\m} & = - (-1)^{\| \varphi \alpha^{-1} \| \| \alpha \varphi^{\op} \|} (-1)^{| \varphi | | \varphi ^{\op} | } \alpha \varphi^{\op} \\
& = - (-1)^{(| \varphi | + \ci{\ta{\varphi}}) ( | \varphi^{\op} | + \ci{\st{\varphi^{\op}}})} (-1)^{| \varphi | | \varphi ^{\op} | } \alpha \varphi^{\op} \\
& = \begin{cases} -1 & \text{if } \ta\varphi \neq i \\ (-1)^{| \varphi | + | \varphi^{\op} |} & \text{if } \ta\varphi = i \end{cases} \quad \cdot \alpha \varphi^{\op} \\
& = - (-1)^{\ci{\ta\varphi} (d+1)} \alpha \varphi^{\op}
\end{align*}
Thus we obtain
\begin{align*}
\partial_{(\varphi \alpha^{-1})^{\op\m}} \dec_{\rm cyc} W & = - (-1)^{\ci{\ta\varphi} (d+1)} \partial_{\alpha \varphi^{\op}} \dec_{\rm cyc} W \\
& = - \underbrace{(-1)^{\ci{\ta\varphi} (d+1)} (-1)^{\ci{\ta\varphi} d}}_{= (-1)^{\ci{\ta\varphi}}} \partial_{\alpha \varphi^{\op}} \dec W \\
& = - (-1)^{\ci{\ta\varphi}} \dec (\partial_{\varphi^{\op}} W / \alpha) \\ & \qquad + (-1)^{\ci{\ta\varphi}} (\dec \red \partial_{\varphi^{\op}} W) \alpha^{-1}
\end{align*}
and
\begin{align*}
& \partial_{(\varphi \alpha^{-1})^{\op\m}} \sum_{\beta, \psi} \beta \dec(\psi \psi^{\op}) \beta^* = - (-1)^{\ci{\ta\varphi} (d+1)} \partial_{\alpha \varphi^{\op}} ( \sum_{\beta, \psi} \beta \dec(\psi \psi^{\op\m}) \beta^*)  \\
& \qquad = - (-1)^{\ci{\ta\varphi} (d+1)} \partial_{\alpha \varphi^{\op}} ( \alpha \varphi^{\op} \varphi^{\op \op\m} \alpha^* ) \\
& \qquad \qquad + \left[ (-1)^{\ci{\ta\varphi} (d+1)} \partial_{\alpha \varphi^{\op}} (\sum_{\beta} \beta \varphi \alpha^{-1} \alpha \varphi^{\op\m} \beta^*) \right]
\end{align*}
where the second summand only appears if $\varphi$ is a loop.

Now, since $\ta{\varphi^{\op}} = i$, we have
\[ \varphi^{\op\op\m} = (-1)^{d+1} \varphi^{\op\op} = (-1)^{d + |\varphi| |\varphi^{\op}|} \varphi, \]
and if $\varphi$ is a loop then
\[ \varphi^{\op\m} = (-1)^{d+1} \varphi^{\op}. \]
Inserting above we obtain
\begin{align*}
& = - (-1)^{\ci{\ta\varphi} (d+1)} (-1)^{d + |\varphi| |\varphi^{\op}|} \partial_{\alpha \varphi^{\op}} ( \alpha \varphi^{\op} \varphi \alpha^* ) \\
& \qquad + \left[ (-1)^{\ci{\ta\varphi} (d+1)} (-1)^{d+1} \partial_{\alpha \varphi^{\op}} (\sum_{\beta} \beta \varphi \alpha^{-1} \alpha \varphi^{\op} \beta^*) \right] \\
& = - (-1)^{\ci{\ta\varphi} (d+1)} (-1)^{d + |\varphi| |\varphi^{\op}|} (-1)^{\| \alpha \varphi^{\op} \| \| \varphi \|} \varphi \alpha^* \\
& \qquad + \left[ \sum_{\beta} \beta^* \beta \varphi \alpha^{-1}  \right] \\
& = (-1)^{ \| \varphi \|} \varphi \alpha^* + \left[ \sum_{\beta} \beta^* \beta \varphi \alpha^{-1}  \right] \\
\end{align*}
Summing up the two halves, and comparing to the first two formulas (depending on whether $\varphi$ ends in $i$ or not), we obtain the second last formula of the lemma.

The calculations used to verify the last equation are similar (and even slightly simpler) to the ones for the second last equation.
\end{proof}

Let $i$ be a vertex of the algebra.

\begin{theorem}
Let $(Q, W)$ be a quiver with potential. Then the derived endomorphism ring of the silting mutation of $(k\bar{Q}, d)$ in $i$ is the Ginzburg algebra of $\mu_i(Q, W)$.
\end{theorem}

\begin{proof}
Again we are in the setup of Theorem~\ref{thm.main_result}.

The mutated quiver $M$ as described in Section~\ref{sec.mutation} has arrows as follows:

\noindent
Step 1: Rotation of arrows
\begin{itemize}
 \item all arrows not involving $i$ just like $\bar{Q}$;
 \item for every arrow $\alpha \colon i \to x \in Q$ of degree $0$ an arrow $\alpha^* \colon x \to i$ of degree $0$;
 \item for every arrow $\varphi \colon i \to x \in Q$ of positive degree the degree of $\varphi$ gets decreased by $1$;
 \item for every arrow $\varphi \colon x \to i \in Q$ the degree of $\varphi$ gets increased by $1$;
 \item all loops in $i$ (including $\ell_i$) keep their degree;
\end{itemize}
Step 2: Compositions
\begin{itemize}
\item for any two arrows $\alpha \colon i \to x \in Q$ and $\varphi \colon y \to i \in Q$, with $| \alpha | = 0$ the composition $\alpha \varphi$ of degree $| \varphi | - \ci{y} $;
\item for any two arrows $\alpha \colon i \to x \in Q$ the composition $\alpha \ell_i \colon i \to x$ of degree $d - 2$;
\end{itemize}
Step 3: Anti-compositions
\begin{itemize}
\item for any two arrows $\alpha \colon i \to x \in Q$ and $\varphi \colon i \to y \in Q$, with $| \alpha | = 0$ and $| \varphi | > 0$ the anti-composition $\varphi \alpha^{-1}$ of degree $\| \varphi \| + 1$;
\item for every arrow $\alpha \colon i \to x$ in $Q$ with $|\alpha| = 0$ an anti-composition arrow $\ell_i \alpha^{-1} \colon x \to i$ of degree $d$;
\end{itemize}
Step 4
\begin{itemize}
 \item for every loop $\varphi$ at $i$ in $Q$, and any two arrows $\alpha, \beta$ starting in $i$ with $| \alpha | = | \beta | = 0$, a new arrow $\alpha \varphi \beta^{-1}$ of degree $| \varphi |$; (note that $\alpha = \beta$ is allowed here.)
\item for any two arrows $\alpha, \beta$ starting in $i$ with $| \alpha | = | \beta | = 0$, a new arrow $\alpha \ell_i \beta^{-1}$ of degree $d-1$.
\end{itemize}
In steps 2 to 4 above it might seem artificial to distinguish the loops $\ell_i$ from arrows in $Q$. However, in the further calculation these two cases have to be treated separately.
(The author found it helpful to look at the quivers in the previous section (where $d = 3$), since they illustrate reasonably well what happens in general.)

Most of the arrows -- more precisely arrows from the original quiver, formal compositions $\alpha \varphi$ as well as anti-compositions $\varphi \alpha^{-1}$ and the combinations $\alpha \varphi \beta^{-1}$ -- appear in $M$ just like in $\mu_i(Q)$. Moreover, comparing Definition~\ref{def.new_diff} to Lemma~\ref{lem.der_of_mutated} we see that also the differential already matches exactly.

\medskip
It remains to deal with arrows involving the special loops $\ell_j$.

We observe
\[ \partial( \ell_i \alpha^{-1} ) = \partial(\ell_i) \alpha^{-1} - (-1)^d \ell_i \alpha^* + \underbrace{(d \ell_i) / \alpha}_{= \alpha^{\op}} . \]
Therefore, up to quasi-isomorphism, the arrows $\ell_i \alpha^{-1}$ and $\alpha^{\op}$ cancel.

Similarly the arrows $\alpha \ell_i \beta^{-1}$ and $\alpha \beta^{\op}$ cancel.

We identify the arrows $\alpha \ell_i$ in $M$ with $(-1)^d (\alpha^*)^{\op\m}$ in $\mu_i(Q)$. The reason for this choice of sign is that
\begin{align*}
\partial(\alpha \ell_i) & = \alpha \dec \red d \ell_i = \alpha \sum_{\varphi} \dec \varphi \varphi^{\op} = (-1)^{d+1} \alpha \sum_{\varphi} \dec \varphi \varphi^{\op\m} \\
& = - (-1)^d \partial_{\alpha^*} W\m = (-1)^d \partial_{(\alpha^*)^{\op\m \op\m}} W\m.
\end{align*}

It now only remains to check that the loops $\ell_j$ have the ``correct" differential. We start by considering the case $j \neq i$. The calculation is reasonably straight-forward, but lengthy, the reason being that all the different types of composition or anti-composition arrows will appear in the differential of $\ell_j$. We start by splitting it up into several parts:
\begin{align*}
\partial( \ell_j ) & = \dec d \ell_j = \dec \sum_{\varphi \colon ? \smash{\to} j} \varphi \varphi^{\op} \\
& = \sum_{\substack{ \varphi \colon ? \smash{\to} j \\ ? \neq i }} \varphi \varphi^{\op} + \sum_{\substack{\varphi \colon i \smash{\to} j \\ | \varphi | > 0 }} \varphi \varphi^{\op} - \sum_{\substack{\varphi \colon i \smash{\to} j \\ | \varphi | > 0 }} \sum_{\alpha} (\varphi \alpha^{-1}) (\alpha \varphi^{\op}) \\ & \qquad + \sum_{\substack{\alpha \colon i \smash{\to} j \\ | \alpha | = 0 }} \alpha \alpha^{\op}
\end{align*}
Noting that $(\varphi \alpha^{-1})^{\op\m} = - \alpha \varphi^{\op}$ we see that the first three sums are already as desired, and we are left with the final one. To analyse that one recall that we cancelled the arrows $\alpha \alpha^{\op}$ for $\alpha \ell_i \alpha^{-1}$ since
\[ \partial( \alpha \ell_i \alpha^{-1}) = \partial(\alpha \ell_i) \alpha^{-1} + (-1)^{\| \alpha \ell_i \|} \alpha \ell_i \alpha^* + \alpha \alpha^{\op}. \]
This means that we in fact identify
\[ \alpha \alpha^{\op} = - \partial(\alpha \ell_i) \alpha^{-1} - (-1)^{\| \alpha \ell_i \|} \alpha \ell_i \alpha^*. \]
Developing the right side we obtain
\[ \alpha \alpha^{op} = - \sum_{\varphi \colon ? \smash{\to} i} \alpha \varphi \underbrace{\varphi^{\op} \alpha^{-1}}_{\mathclap{= - (\alpha \varphi)^{\op\m}}} + \sum_{\varphi \colon i \smash{\to} i} \sum_{\beta} (\alpha \varphi \beta^{-1}) (\beta \varphi^{\op} \alpha^{-1}) - \underbrace{(-1)^d \alpha \ell_i}_{\widehat{=} (\alpha^*)^{\op\m}} \alpha^*. \]
So the differential of $\ell_j$ is exactly as it should be for a Ginzburg algebra.

\medskip
Finally we calculate
\begin{align*}
\partial(\ell_i) & = \sum_{\alpha} \alpha^* \alpha \ell_i - \dec \red d \ell_i \\
& = (-1)^d \sum_{\alpha} \alpha^* (\alpha^*)^{\op\m} - \sum_{\substack{ \varphi \colon ? \smash{\to} i \\ | \varphi^{\op} | > 0}} \varphi \varphi^{\op} + \sum_{\varphi \colon i \smash{\to} i} \sum_{\alpha} \varphi \alpha^{-1} \alpha \varphi^{\op}.
\end{align*}
In the second sum $\varphi$ ends in $i$, so $\varphi^{\op\m} = - (-1)^d \varphi^{\op}$. In the final sum we note that $(\varphi \alpha^{-1})^{\op\m} = (-1)^d \alpha \varphi^{\op}$.

So, if we identify $\ell_i$ in $M$ with $(-1)^d \ell_i$ in the mutated Ginzburg algebra, the differential matches.
\end{proof}

\section{Proof of the main result} \label{sec.proof}

Now we give a proof of Theorem~\ref{thm.main_result}. That is, we calculate the endomorphism ring of the silting object obtained by silting mutating in vertex $i$, and show that it is quasi-isomorphic to the dg quiver $(M, \partial)$ described in Section~\ref{sec.mutation}. Here we will use the fact that $(M, \partial)$ is in fact a dg quiver, which will be proven in Section~\ref{sec.is_diff}.

\medskip
When silting mutating in $i$, we replace $P_i$ by $P_i^*$, the cone of it's approximation by the other $P_j$. Explicitly that is
\[ P_i^* = \Cone(P_i \to \bigoplus_{\alpha \in A} P_{\tail(\alpha)}). \]
Clearly morphisms between the unmutated summands don't change (and neither does their differential).

We observe that, disregarding the differential, $P_i^* = P_i[1] \oplus \bigoplus_{\alpha \in A} P_{\tail(\alpha)}$. We denote by $\alpha^* \colon P_{\tail(\alpha)} \smash{\to} P_i^*$ and $\alpha_* \colon P_i^* \smash{\to} P_{\tail(\alpha)}$ the split mono- and epimorphisms corresponding to the summands $P_{\tail(\alpha)}$, and similarly $M \colon P_i \smash{\to} P_i^*$ and $E \colon P_i^* \smash{\to} P_i$ the split mono- and epimorphism corresponding to the summand $P_i$. Note that the $\alpha^*$ and $\alpha_*$ are of degree $0$, while $M$ is of degree $1$ and $E$ is of degree $-1$.

The endomorphism ring $\Gamma$ of $P_i^* \oplus \bigoplus_{j \neq i} P_j$ is generated by
\begin{itemize}
\item the arrows $\overline{\varphi}$, where $\varphi$ is an arrow of $Q$, and
\[ \overline{\varphi} = \begin{cases} \varphi & \text{if } \start(\varphi) \neq i \neq \tail(\varphi) \\ M \varphi & \text{if } \start(\varphi) \neq i = \tail(\varphi) \\ \varphi E & \text{if } \start(\varphi) = i \neq \tail(\varphi) \\ M \varphi E & \text{if } \start(\varphi) = i = \tail(\varphi) \end{cases} \]
\item the arrows $\alpha^*$ and $\alpha_*$.
\end{itemize}
subject to the relations
\begin{itemize}
 \item $\alpha_* \alpha^* = \id_{P_{\tail(\alpha)}}$;
 \item $\alpha_* \varphi = 0$ for any arrow $\varphi \neq \alpha^*$;
 \item $\varphi \alpha^* = 0$ for any arrow $\varphi \neq \alpha_*$.
\end{itemize}
We denote the quiver given by the arrows $\overline{\varphi}, \alpha_*,$ and $\alpha^*$ by $\overline{Q}$.

We now calculate the differential of these morphisms. The first thing to note is that the differential in $P_i^*$ is just given by the matrix
\[ \begin{pmatrix} -d_{P_i}[1] & 0 & \cdots & 0 \\ \alpha_1 & d_{P_{\tail(\alpha_1)}} & \ddots & \vdots \\ \vdots && \ddots & 0 \\ \alpha_s & 0 & \cdots & d_{P_{\tail(\alpha_s)}} \end{pmatrix} \]
We note that the $\alpha^*$, as well as $E$, commute with differentials, that is that $d \alpha^* = 0$ and $dE = 0$. Moreover
\begin{align*}
& d \alpha_* = d_{P_{\tail(\alpha)}} \alpha_* - \alpha_* d_{P_i^*} = - \overline{\alpha}, && \text{and} \\
& dM = d_{P_i^*}[-1] M + M d_{P_i} = \sum_{\alpha \in A} \alpha^* \alpha.
\end{align*}
Now, with the graded Leibniz rule, we obtain
\[ d \overline{\varphi} = \begin{cases} \overline{d \varphi} & \text{if } \tail(\overline{\varphi}) \neq i^* \\ - \overline{d\varphi} + \sum_{\alpha \in A} \alpha^* \overline{\alpha \varphi} & \text{if } \tail(\overline{\varphi}) = i^* \end{cases}. \]
(Note that the notion $\overline{\cdot}$ naturally extends to linear combinations of paths in $Q$, so writing the above makes sense.)

\subsection{Removing degree $-1$}

In the above construction we may note that among the ``new arrows'' $\overline{\varphi}$, some, more precisely those of the form $\overline{\alpha} = \alpha E$, have degree $-1$, while our endomorphism ring should be concentrated in non-negative degrees. The reason for this is simple: these arrows of degree $-1$ ``cancel'' against the parallel arrows $\alpha_*$ of degree $0$. The aim of this subsection is to make this statement precise.

\begin{remark}
A first na{\"i}ve idea might be (read: the author's first idea was) to factor out the ideal generated by $\alpha^*$ and $\overline{\alpha}$. However, given the relation $\alpha_* \alpha^* = \id_{P_{\tail(\alpha)}}$, this is not a good idea, and we in fact proceed by finding a suitable subalgebra not involving these arrows. Also here one has to be careful: the graded subalgebra generated by paths not involving these arrows is not closed under the differential, so we need to apply a ``smoothening procedure" (see Construction~\ref{const.hat} below) to obtain a dg subalgebra.
\end{remark}

We first need to apply a linear transformation to our generators:

\begin{definition}
A path in the quiver $\overline{Q}$ will be called \emph{basic} if it does not contain $\alpha_*$ -- except possibly as the first morphism, and doesn't contain $\alpha^*$ -- except possibly as the last morphism. In other words a path is basic if it cannot be further simplified by applying the relations.

One immediately sees that the basic paths form a basis for our algebra.

\medskip
The following two constructions are essentially the same as in Construction~\ref{const.red_dec}:

For a basic path $p$, and an arrow $\overline{\alpha}$, we set
\[ p / \overline{\alpha} = \begin{cases} p' & \text{if } p = p' \overline{\alpha} \\ 0 & \text{otherwise.} \end{cases}. \]
This extends to a linear map $- / \overline{\alpha}$ on the algebra.

For a basic path $p$ we set
\[ \red(p) = \begin{cases} 0 & \text{if } p = p' \overline{\alpha} \text{ for some } \alpha \in A \\ p & \text{otherwise.} \end{cases}. \]
This also extends to a linear map on the algebra, which essentially removes all paths starting with an arrow $\overline{\alpha}$ from a linear combination. We call a basic path \emph{reduced} if it does not start with any $\overline{\alpha}$.
\end{definition}

\begin{construction} \label{const.hat}
Now for a reduced basic path $p$ in $\overline{Q}$, we set
\[ \widehat{p} = p + (-1)^{|p|} \sum_{\alpha \in A} ( dp ) / \overline{\alpha} \; \alpha_*. \]
\end{construction}

\begin{lemma} \label{lem.reduction_preserves_diff}
Let $p$ be a (linear combination of) reduced basic paths. Then
\[ d \widehat{p} = \widehat{\red(dp)}. \]
\end{lemma}

\begin{proof}
Clearly
\[ dp = \red(dp) + \sum_{\alpha \in A} dp / \overline{\alpha} \; \overline{\alpha} \text{ and } d(dp) = 0, \]
whence
\[ d(\red(dp)) = - \sum_{\alpha \in A} d( (dp) / \overline{\alpha} \; \overline{\alpha}) = - \sum_{\alpha \in A} d( (dp) / \overline{\alpha} ) \overline{\alpha} \]
and
\[ d(\red(dp)) / \overline{\alpha} = - d ( (dp) / \overline{\alpha}). \]
It follows that
\begin{align*}
\widehat{\red(dp)} & = \red(dp) + (-1)^{|\red(dp)|} \sum_{\alpha \in A} d(\red(dp)) / \overline{\alpha} \; \alpha_* \\
& = \red(dp) - (-1)^{| dp |} \sum_{\alpha \in A} d( (dp) / \overline{\alpha}) \alpha_* \\
& = \red(dp) + (-1)^{|p|} \sum_{\alpha \in A} d( (dp) / \overline{\alpha}) \alpha_* .
\end{align*}
On the other hand
\begin{align*}
d \widehat{p} & = dp + (-1)^{|p|} \sum_{\alpha \in A} d((dp) / \overline{\alpha} \; \alpha_*) \\
& = dp + (-1)^{|p|} \sum_{\alpha \in A} \left( d((dp) / \overline{\alpha}) \alpha_* + (-1)^{|dp / \overline{\alpha}|} (dp) / \overline{\alpha} \;  (- \overline{\alpha}) \right) \\
& = dp - \sum_{\alpha \in A} (dp) / \overline{\alpha} \; \overline{\alpha} +  (-1)^{|p|} \sum_{\alpha \in A} d((dp) / \overline{\alpha}) \alpha_* \\
& \qquad \qquad \text{(here we use $|dp/\overline{\alpha}| = |dp| - |\overline{\alpha}| = |p| - 1 - (-1) = |p|$)} \\
& = \red(dp) + (-1)^{|p|} \sum_{\alpha \in A} d( dp / \overline{\alpha} ) \alpha_*. \qedhere
\end{align*}
\end{proof}

\begin{theorem}
Let $\Gamma_{\rm red}$ be the subvector space of $\Gamma$ generated by the $\widehat{p}$, where $p$ is a reduced basic path in $\overline{Q}$.

Then $\Gamma_{\rm red}$ is a dg subalgebra of $\Gamma$, and the inclusion $\Gamma_{\rm red} \embed \Gamma$ is a quasi-isomorphism.
\end{theorem}

\begin{proof}
Clearly for reduced paths the construction $\widehat{\cdot}$ only affects the rightmost arrow, so
\[ \widehat{p} \widehat{q} = \widehat{ \, \widehat{p} q \, } \in \Gamma_{\rm red} \]
for any non-trivial reduced path $q$. Therefore $\Gamma_{\rm red} $ is a subalgebra of $\Gamma$.

It follows from Lemma~\ref{lem.reduction_preserves_diff} that $\Gamma_{\rm red}$ it is also closed under the differential.

Now consider the quotient $\Gamma / \Gamma_{\rm red}$. Clearly it is generated by reduced paths starting with either $\alpha_*$ or $\overline{\alpha}$. Since $d \alpha_* = - \overline{\alpha}$ is is easy to see that this cokernel is acyclic.
\end{proof}

Note that as a quiver with relations, $\Gamma_{\rm red}$ is given by the following arrows and relations. Here, by abuse of notation, we write $\widehat{\varphi}$ instead of $\widehat{\overline{\varphi}}$ for arrows $\varphi \in Q$.
\begin{itemize}
\item an arrow $\widehat{\varphi}$ for any arrow $\varphi$ of $Q$ which is not in $A$;
\item the arrows $\alpha^*$ as in $\overline{Q}$;
\item an arrow $\widehat{\alpha \varphi}$ for any $? \to^{\overline{\varphi}} i \to^{\overline{\alpha}} ?$ in $\overline{Q}$.
\end{itemize}
subject to the relations
\begin{align*}
\widehat{\varphi} \alpha^* & = \underbrace{\overline{\varphi} \alpha^*}_{ = 0} + (-1)^{|\overline{\varphi}|} \sum_{\beta \in A} (d \overline{\varphi}) / \overline{\beta} \; \beta_* \alpha^* \\
& = (-1)^{|\overline{\varphi}|} { (d \overline{\varphi}) / \overline{\alpha}} \\
& = \begin{cases} - (-1)^{|\varphi|} \overline{(d \varphi) / \alpha } & \text{ if } \ta{\varphi} \neq i \\ (-1)^{|\varphi|} (- \overline{d \varphi} + \sum_{\beta \in A} \beta^* \overline{\beta \varphi} ) / \overline{\alpha} & \text{ if } \ta{\varphi} = i \end{cases} \\
& = - (-1)^{|\varphi|} \widehat{ (d \varphi) / \alpha } \\
\intertext{and}
\widehat{\alpha \varphi} \beta^* & = - (-1)^{|\varphi|} \widehat{ \alpha \; (d \varphi) / \beta }.
\end{align*}

\subsection*{Cancelling anti-compositions}

As a last step towards the proof of Theorem~\ref{thm.main_result}, we simplify the resulting algebra $kM$.

\begin{proposition}
The dg-algebra $(kM, \partial)$ described in Section~\ref{sec.mutation}, is quasi-isomorphic to its dg-quotient algebra generated by arrows of the forms
\[ \varphi, \alpha \varphi, \alpha^* \]
subject to the relations
\[ \varphi \alpha^* = - (-1)^{|\varphi|} (d \varphi) / \alpha \text{, and } \alpha \varphi \beta^* = - (-1)^{|\varphi|} \alpha (d \varphi) / \beta. \]
\end{proposition}

\begin{proof}
We observe that $\partial( \varphi \alpha^{-1}) = (-1)^{\| \varphi \|} \varphi \alpha^* - (-1)^{\ci{ \ta{\varphi}}} d \varphi / \alpha$ up to terms $\psi \beta^{-1}$ with $\psi$ of strictly smaller degree. Thus, cancelling the $\varphi \alpha^{-1}$ for their differential, we obtain relations
\[ (-1)^{\| \varphi \|} \varphi \alpha^* - (-1)^{\ci{ \ta{\varphi}}} d \varphi / \alpha. \]
Now note that
\[ \| \varphi \| = |\varphi| - \underbrace{\ci{ \st{\varphi}}}_{=1} + \ci{ \ta{\varphi}}, \]
so the above relation gives the first one of the proposition.

The argument involving arrows $\alpha \varphi \beta^{-1}$ and leading to the second relation of the proposition is similar.
\end{proof}

\section{Well-definedness of the new differential} \label{sec.is_diff}

The aim of this section is to prove the following theorem.

\begin{theorem} \label{thm.is_diff}
Let $\partial$ be the map on $kM$ as defined in Definition~\ref{def.new_diff}. Then
\[ \partial^2 = 0. \]
\end{theorem}

This is essentially checked by brute force calculation. We first prepare some intermediate result, which essentially says that (versions of) the formulas in Definition~\ref{def.new_diff} also hold for paths, not only arrows.

\begin{lemma} Let $p$ be any linear combination of paths in $Q$ not starting with arrows in $A$. Then
\[ \partial( \dec p ) = \begin{cases} \dec ( \red (dp)) & \text{ if } \ta{p} \neq i \\ \sum_{\alpha \in A} \alpha^* \alpha \dec(p) - \dec ( \red (dp)) & \text{ if } \ta{p} = i \end{cases}, \]
and
\[ \partial( \dec(p) \alpha^{-1} ) = \partial( \dec p ) \alpha^{-1} + (-1)^{\| p \|} \dec(p) \alpha^* - (-1)^{\ci{ \ta{\varphi}}} \dec (dp / \alpha). \]
\end{lemma}

\begin{proof}
Clearly it suffices to check this for $p$ a path. Moreover, if $p$ is an arrow $\varphi$ or of the form $\alpha \varphi$ then the claims hold by definition. We now use induction on the length of the path $p$. Write $p$ as $p = p_2 p_1$ with $p_1$ and $p_2$ paths of shorter length, not starting with arrows in $A$.

We start by checking the first claim, assuming that both claims hold for shorter paths. Assume first that $\ta{p_1} = \st{p_2} =i$. In that case $\dec(p) = \dec(p_2) \Delta \dec(p_1)$, and we have
\begin{align*}
\partial(\dec(p)) & = \partial(\dec(p_2) \dec(p_1)) - \partial ( \sum_{\alpha \in A} \dec (p_2) \alpha^{-1} \alpha \dec(p_1) ) \\
& = \partial(\dec(p_2)) \dec(p_1) \\
& \qquad + (-1)^{\| p_2 \|} \dec(p_2) \Big( \sum_{\alpha \in A} \alpha^* \alpha \dec(p_1) - \dec(\red(dp_1)) \Big) \\
& \qquad - \sum_{\alpha \in A} \Big( \partial( \dec( p_2)) \alpha^{-1} + (-1)^{\|p_2\|} \dec(p_2) \alpha^* \\
& \qquad \qquad - (-1)^{\ci{\ta{p_2}}} \dec((dp_2) / \alpha) \Big) \alpha \dec(p_1) \\
& \qquad - (-1)^{\| p_2 \| + 1} \sum_{\alpha \in A} \dec(p_2) \alpha^{-1} \alpha \dec( \red( d p_1 )) \\
& = \partial(p_2) \Delta \dec( p_1) - (-1)^{\| p_2 \|} \dec(p_2) \Delta \dec( \red( d p_1)) \\
& \qquad + (-1)^{\ci{\ta{p_2}}} \dec((dp_2) / \alpha) \alpha \dec(p_1) \\
\intertext{Inserting the formula for $\partial(\dec(p_2))$ into the two above (and allowing ourselves to use square brackets around terms that only appear in case $\ta{p} = i$) this gives}
& = \Big( \Big[ \sum \alpha^* \alpha \dec(p_2) \Big] + (-1)^{\ci{ \ta{p}}} \dec(\red(dp_2)) \Big) \Delta \dec( p_1) \\
& \qquad + (-1)^{| p_2 | + \ci{ \ta{p}}} \dec(p_2) \Delta \dec( \red( d p_1)) \\
& \qquad + (-1)^{\ci{ \ta{p}}} \dec((dp_2) / \alpha) \alpha \dec(p_1) \\
& = \Big[ \sum \alpha^* \alpha \dec(p) \Big] \\
& + (-1)^{\ci{ \ta{p}}} \dec \Big( (\underbrace{\red(dp_2) + \sum_{\alpha \in A} (d p_2) / \alpha \; \alpha}_{= dp_2}) p_1 + (-1)^{|p_2|} p_1 dp_2 \Big) \\
& = \Big[ \sum \alpha^* \alpha \dec(p) \Big] + (-1)^{\ci{ \ta{p}}} \dec( \red( dp))
\end{align*}
For the case $\st{p_2} = \ta{p_1} \neq i$ we first obtain
\begin{align*}
\partial( \dec(p) ) & = \partial(\dec(p_2)) \dec(p_1) + (-1)^{\| p_2 \|} \dec(p_2) \dec( \red( dp_1)), \\
\intertext{and, inserting for $\partial(\dec(p_2))$}
& = \Big( \Big[ \sum \alpha^* \alpha \dec(p_2) \Big] + (-1)^{\ci{ \ta{p}}} \dec(\red(dp_2)) \Big) \dec(p_1) \\
& \qquad + (-1)^{|p_2| + \ci{ \ta{p}}} \dec(p_2) \dec(\red(d p_1)) \\
& = \Big[ \sum \alpha^* \alpha \dec(p) \Big] + (-1)^{\ci{ \ta{p}}} \dec( \red( dp))
\end{align*}

Now we check the second claim. We only check the case that $\st{p_2} = \ta{p_1} = i$. The calculation in the other case is similar but simpler.
\begin{align*}
\partial( \dec(p) \alpha^{-1} ) & = \partial( \dec(p_2) \Delta \dec(p_1) \alpha^{-1} ) \\
& = \partial( \dec(p_2) ) \dec(p_1) \alpha^{-1} + (-1)^{\|p_2\|} \dec(p_2) \partial(\dec(p_1) \alpha^{-1}) \\
& \qquad - \sum_{\beta \in A} \Big( \partial(\dec(p_2) \beta^{-1}) \beta \dec(p_1) \alpha^{-1} \\
& \qquad \qquad - (-1)^{\|p_2\|} \dec(p_2) \beta^{-1} \partial( \beta \dec(p_1) \alpha^{-1}) \Big) \\
& =\partial( \dec(p_2) ) \dec(p_1) \alpha^{-1} + (-1)^{\|p_2\|} \dec(p_2) \partial(\dec(p_1)) \alpha^{-1} \\
& \qquad - \sum_{\beta \in A} \Big( \partial(\dec(p_2) \beta^{-1}) \beta \dec(p_1) \alpha^{-1} \\
& \qquad \qquad - (-1)^{\|p_2\|} \dec(p_2) \beta^{-1} \partial( \beta \dec(p_1)) \alpha^{-1} \Big) \\
& \qquad + (-1)^{\| p \|} \dec(p_2) \dec(p_1) \alpha^* + (-1)^{\|p_2\|} \dec(p_2) \dec( d p_1 / \alpha ) \\
& \qquad - \sum_{\beta \in A} \Big( (-1)^{\| p \|} \dec(p_2) \beta^{-1} \beta \dec(p_1) \alpha^* \\
& \qquad \qquad + (-1)^{\| p_2 \|} \dec{p_2} \beta^{-1} \beta \dec( d p_1 / \alpha) \Big) \\
& = \partial( \dec(p)) \alpha^{-1} + (-1)^{\| p \|} \dec(p) \alpha^* + (-1)^{\|p_2\|} \dec(\underbrace{p_2 (d p_1) / \alpha }_{= (-1)^{\| p_2 \|} dp / \alpha}) \\
& = \partial( \dec(p)) \alpha^{-1} + (-1)^{\| p \|} \dec(p) \alpha^* + (-1)^{\|p_2\|} \dec (d p) / \alpha ) \qedhere
\end{align*}
\end{proof}

\begin{proposition}
$\partial$ defines a differential, that is $\partial^2 = 0$.
\end{proposition}

\begin{proof}
Since $\partial$ follows the Leibnitz rule, it suffices to check that $\partial^2$ vanishes on arrows.

For arrows $\varphi$ coming from the original quiver with $\ta{\varphi} = i$ we have
\begin{align*}
\partial^2(\varphi) & = \partial(\sum_{\alpha \in A} \alpha^* \alpha \varphi - \dec( \red( d \varphi))) \\
& = \sum_{\alpha \in A} \alpha^* \alpha \dec(\red(\varphi)) - \sum_{\alpha \in A} \alpha^* \alpha \dec(\red(d \varphi)) + \dec( \red( d (\red( d \varphi)))) \\
& = \dec( \red( d (\red( d \varphi))))
\end{align*}
Clearly this also holds when $\ta{\varphi} \neq i$.

Now we observe that $\red( d \varphi)$ and $d \varphi$ only differ by a linear combination of paths starting in arrows in $A$. Therefore $d( \red( d \varphi))$ also is a linear combination of such paths, and the above vanishes.

For arrows of the form $\alpha \varphi$ the same argument works.

For the arrows $\alpha^*$ we already have $\partial(\alpha^*) = 0$.

For arrows $\varphi \alpha^{-1}$ we calculate
\[ \partial^2(\varphi \alpha^{-1}) = \partial\Big( \partial(\varphi) \alpha^{-1} + (-1)^{\| \varphi \|} \varphi \alpha^* - (-1)^{\ci{ \ta{\varphi}}} \dec( (d \varphi) / \alpha) \Big) \]
Similar to the calculation for $\partial^2(\varphi)$ above one can check that
\[ \partial ( \partial( \varphi) \alpha^{-1} ) = (-1)^{\| \partial(\varphi) \|}  \partial( \varphi ) \alpha^* - (-1)^{\ci{ \ta{\varphi}}} \dec(( d( \red (d \varphi) )) / \alpha ). \]
Thus we obtain
\begin{align*}
\partial^2(\varphi \alpha^{-1}) & = - (-1)^{\ci{ \ta{\varphi}}} ( \partial( \dec( (d \varphi) / \alpha)) + \dec(( d( \red (d \varphi) )) / \alpha ) ) \\
& = - (-1)^{\ci{ \ta{\varphi}}} \dec\Big( d( (d \varphi) / \alpha) + (d( \red ( d \varphi))) / \alpha \Big)
\end{align*}
Finally we note that
\[ 0 = d^2 \varphi = d(\red( d \varphi)) + \sum_{\alpha \in A} d( (d \varphi) / \alpha \; \alpha ). \]
Restricting to paths starting with $\alpha$ this gives that $\partial^2( \varphi \alpha^{-1}) = 0$.

Again the proof for arrows $\alpha \varphi \beta^{-1}$ is the same calculation.
\end{proof}

\end{document}